\documentclass[12pt]{article}

\usepackage[T1]{fontenc}
\usepackage{lmodern}
\usepackage[margin=1.15in]{geometry}
\usepackage{amsmath,amssymb,amsthm,mathtools}
\usepackage{microtype}
\usepackage[hidelinks]{hyperref}

\newtheorem{theorem}{Theorem}[section]
\newtheorem{lemma}[theorem]{Lemma}
\newtheorem{corollary}[theorem]{Corollary}
\newtheorem{conjecture}[theorem]{Conjecture}

\newtheorem{claim}[theorem]{Claim}

\title{Large induced subgraphs with $k$ vertices of maximum degree}
\author{Zhen Liu\footnote{Email: 1552580575@qq.com}, ~Qinghou Zeng\footnote{Research supported by National Key R\&D Program of China (Grant No. 2023YFA1010202) and National Natural Science Foundation of China (Grant No. 12371342). Email: zengqh@fzu.edu.cn (Corresponding
author)}\\
{\small Center for Discrete Mathematics, Fuzhou University, Fujian, 350003, China}}
\date{}
\begin{document}

\maketitle

\begin{abstract}
We prove that, for every integer \(k\ge 2\), there exists a constant \(c_k>0\) such that every graph on \(n\ge R(k,k)\) vertices with maximum degree \(\Delta\) contains an induced subgraph on at least $n-c_k\sqrt{\Delta}$ vertices whose maximum degree is attained by at least \(k\) vertices. This confirms a conjecture of Caro and Yuster in strong form.
\end{abstract}

\section{Introduction}\label{Intro}

All graphs considered are finite, simple and undirected. The \emph{repetition number} $\operatorname{rep}(G)$ of a graph $G$ is the maximum multiplicity of a vertex degree; trivially $\operatorname{rep}(G)\ge2$ for $|V(G)|\ge2$. Repeated degrees have been studied extensively (see, e.g., \cite{BalisterEtAl2013,BollobasScott1997,CaroWest2009}).

Graphs of arbitrarily large order can have $\operatorname{rep}(G)=2$, so it is natural to ask how many vertices must be deleted to force several vertices of the remaining induced subgraph to share the same degree. Caro, Shapira and Yuster~\cite{CaroShapiraYuster2014} proved that for every $k$ there exists a constant $C(k)$ such that every $n$-vertex graph contains an induced subgraph on at least $n-C(k)$ vertices with at least $\min\{k,n-C(k)\}$ vertices of equal degree. Thus, if the common degree need not be maximal, a bounded number of deletions suffices.

Caro and Yuster~\cite{CaroYuster2010} considered the more restrictive problem where the common degree must be the maximum degree of the remaining graph. For $k\ge2$ and an $n$-vertex graph $G$ with $n\ge R(k,k)$, let $f_k(G)$ be the minimum number of vertices to delete so that the resulting induced subgraph has its maximum degree attained by at least $k$ vertices (the Ramsey condition guarantees that $f_k(G)$ is well‑defined). They constructed graphs with $f_2(G)\ge(1-o(1))\sqrt n$ and conjectured $f_k(G)\le c_k\sqrt n$ for every $k$. A stronger version, replacing $n$ by $\Delta(G)$, was proposed by Caro, Lauri and Zarb~\cite{CaroLauriZarb2018}.

\begin{conjecture}[Caro, Lauri and Zarb \cite{CaroLauriZarb2018}]\label{conj:caro-yuster}
For every integer $k\ge2$ there is a constant $a_k$ such that every graph $G$ on at least $R(k,k)$ vertices satisfies $f_k(G)\le a_k\sqrt{\Delta(G)}$.
\end{conjecture}

Caro, Lauri and Zarb~\cite{CaroLauriZarb2018} proved the bound $f_k(G)\le (k-1)\Delta(G)$, which guarantees that the following maximum is well defined.  Set
$g(k,\Delta)=\max\{f_k(G):|V(G)|\ge R(k,k),\ \Delta(G)\le\Delta\}$.
For $\Delta\ge1$, they determined exactly
$g(2,\Delta)=\bigl\lceil(-3+\sqrt{8\Delta+1})/2\bigr\rceil$.
The square‑root order is already necessary for $k=2$, and analogous constructions in the same paper give
$g(k,\Delta)=\Omega_k(\sqrt{\Delta})$ for every fixed $k$.
Gir\~ao and Popielarz~\cite{GiraoPopielarz2018} proved an approximate version: for every $k$ there are constants $b_k,c_k$ such that every graph $G$ contains an induced subgraph $H$, obtained by deleting at most $b_k\sqrt{\Delta(G)}$ vertices, in which $k$ vertices share the same degree and that degree is at least $\Delta(H)-c_k$. Further results for forests and graphs without short cycles appear in F\"urst et al.~\cite{FurstEtAl2020}.
In this paper we fully confirm Conjecture~\ref{conj:caro-yuster}.
\begin{theorem}\label{thm:main}
For every integer $k\ge2$ there exists a constant $A_k$ such that every graph $G$ on at least $R(k,k)$ vertices satisfies $f_k(G)\le A_k\sqrt{\Delta(G)}$.
\end{theorem}

For technical convenience we extend the definition to all graphs. Call $Z\subseteq V(G)$ \emph{successful} if either $|G-Z|<k$ or at least $k$ vertices of $G-Z$ have degree $\Delta(G-Z)$, and let $s_k(G)$ be the minimum size of a successful set. If $|V(G)|\ge R(k,k)$, Ramsey's theorem yields a $k$-vertex clique or independent set, so deleting all but those $k$ vertices gives a successful set of size $|V(G)|-k$, while any set leaving fewer than $k$ vertices needs $|V(G)|-k+1$ deletions. Hence a minimum successful set must satisfy the second condition, giving $s_k(G)=f_k(G)$. Thus it suffices to prove $s_k(G)\le A_k\sqrt{\Delta(G)}$ for every graph $G$.

\paragraph{Notation.}
For a positive integer $q$, set $[q]=\{1,\dots,q\}$.
For a graph $G$, $V(G)$ and $E(G)$ are its vertex set and edge set, $|V(G)|$ is its order,
$d_G(v)$ is the degree of $v$, and $\Delta(G)$ is the maximum degree.
If $A\subseteq V(G)$, then $G[A]$ denotes the subgraph induced by $A$, and $G-A$ means $G[V(G)\setminus A]$.
The neighborhood of $v$ in $G$ is $N_G(v)$; when the graph is clear we simply write $N(v)$,
and for a set $A\subseteq V(G)$ we write $d_A(v)=|N(v)\cap A|$.
The number of edges with both ends in $A$ is $e(A)$, and for disjoint sets $A,B\subseteq V(G)$,
$e(A,B)$ is the number of edges between them.
Suppose that $u$ and $v$ are two vertices of $G$, and let $\mathbf{1}_{uv}=1$ if $uv\in E(G)$ and $\mathbf{1}_{uv}=0$ otherwise.
For a vector $z=(z_1,\dots,z_d)$, $\|z\|_\infty=\max_{1\le i\le d}|z_i|$.
For positive integers $a,b$, the Ramsey number $R(a,b)$ is the smallest integer $m$ such that every graph on $m$ vertices contains a clique of order $a$ or an independent set of order $b$.

\section{Proof of the Main Theorem}
Our proof requires the following lemma.

\begin{lemma}[Caro, Shapira and Yuster~\cite{CaroShapiraYuster2014}]
\label{lem:caro}
For positive integers $r, d, q$ the following holds. 
Any sequence of at least $(\lceil q/r \rceil + 2)(2rd+1)^d$ 
elements of the  $\mathbb{Z}^d \cap [-r,r]^d$
whose sum, denoted by $z$, is in $[-q,q]^d$ 
contains a subsequence of length at most $(\lceil q/r \rceil + 2)(2rd+1)^d$ 
whose sum is $z$.
\end{lemma}
The following lemma differs only very slightly from the $r=1$ case of Lemma~\ref{lem:caro}; for the sake of rigor, we present a short proof.
\begin{lemma}\label{lem:bounded-representation}
For every integer $d\ge1$ there exists a constant $c_1=c_1(d)$ such that the following holds.
Let $v_1,\dots,v_N\in\{-1,0,1\}^d$, set $z=\sum_{i=1}^N v_i$, and suppose $\|z\|_\infty\le b$ for some integer $b\ge0$.
Then there is a subcollection of at most $c_1 b$ vectors whose sum equals $z$.
\end{lemma}
\begin{proof}
If $b=0$, then $z=0$ and the empty subcollection works.
Assume $b\ge1$. If $N\ge (b+2)(2d+1)^d$, then Lemma~\ref{lem:caro}
applied with $r=1$ and $q=b$ yields a subcollection of at most
$(b+2)(2d+1)^d$ vectors summing to $z$. If $N<(b+2)(2d+1)^d$,
taking all $N$ vectors already gives a subcollection of size
$N\le (b+2)(2d+1)^d$. Since $b\ge1$,
$(b+2)(2d+1)^d\le 3(2d+1)^d\,b$.
Thus in either case we obtain a subcollection of at most
$3(2d+1)^d\,b$ vectors summing to $z$, so the lemma holds with
$c_1(d)=3(2d+1)^d$.
\end{proof}
The proof of the following lemma is inspired by Theorem~1 of Gir{\~a}o and Popielarz~\cite{GiraoPopielarz2018}.
\begin{lemma}\label{lem:degree-balancing}
For every integer $k\ge 2$, there exists a constant $c_2 = c_2(k)$ such that the following holds. Let $H$ be a graph, let $b$ be a nonnegative integer, and let $X = \{x_1, \dots, x_k\}$ be a vertex set for which $H[X]$ is a regular graph. Suppose that for all $u, v \in X$, we have $|d_H(u) - d_H(v)| \le b$. Then there exists a set $U \subseteq V(H) \setminus X$ with $|U| \le c_2 b$ such that all vertices in $X$ have equal degree in $H - U$.
\end{lemma}

\begin{proof}
Let $c_2 = c_1$, where $c_1$ is the constant from Lemma~\ref{lem:bounded-representation} with $d = k-1$. For each $v \in V(H) \setminus X$, we define the vector $w_v$ by
\[
w_v(i) = \mathbf{1}_{vx_i } - \mathbf{1}_{vx_k} \quad \text{for } 1 \le i < k.
\]
Thus $w_v \in \{-1, 0, 1\}^{k-1}$. The $i$-th coordinate of $w_v$ records the difference between the contributions of $v$ to the degrees of $x_i$ and $x_k$.

Let $z = \sum_{v \in V(H) \setminus X} w_v$. For each $1 \le i < k$, the $i$-th coordinate of $z$ equals $d_{V(H)\setminus X}(x_i) - d_{V(H)\setminus X}(x_k)$. Since $H[X]$ is regular, we have $d_X(x_i) = d_X(x_k)$ for all $1 \le i < k$. Hence,
\begin{align*}
z(i) &= d_{V(H)\setminus X}(x_i) - d_{V(H)\setminus X}(x_k)
\\&= \bigl(d_{V(H)\setminus X}(x_i) + d_X(x_i)\bigr) - \bigl(d_{V(H)\setminus X}(x_k) + d_X(x_k)\bigr)
= d_H(x_i) - d_H(x_k).
\end{align*}
Since $|d_H(u) - d_H(v)| \le b$ for all $u, v \in X$, it follows that $\|z\|_\infty \le b$.

By Lemma~\ref{lem:bounded-representation}, there exists a set $U \subseteq V(H) \setminus X$ such that $|U| \le c_2b$ and $\sum_{v \in U} w_v = z$. Thus, for
$1\le i<k$,
\begin{align*}
 d_{V(H)\setminus U}(x_i)-d_{V(H)\setminus U}(x_k)
 &=d_H(x_i)-d_H(x_k)
   -\bigl(d_U(x_i)-d_U(x_k)\bigr)\\
 &=d_H(x_i)-d_H(x_k)-\sum_{v\in U}w_v(i)\\
 &=z(i)-z(i)=0.
\end{align*}
The last equality holds because the $i$th coordinate of $z$ is
$d_H(x_i)-d_H(x_k)$ and $\sum_{v\in U}w_v=z$.
The set $U$ is disjoint from $X$, so all vertices of $X$ remain, and
the last display shows that they all have the same degree in $H-U$.
\end{proof}

Let $G$ be a graph on $n$ vertices, and fix an integer $t$ with $0 < t < n$. Among all $(n-t)$-vertex subsets of $V(G)$, choose $S$ such that $G[S]$ has the minimum number of edges, and put $D = V(G) \setminus S$. We use $S$ and $D$ with these meanings whenever this minimum-edge construction is in force.
For $v \in D$ and $u \in S$, replacing $u$ by $v$ cannot reduce the number of edges in the chosen set. Hence
\begin{equation}\label{eq:one-exchange}
e(S \setminus \{u\} \cup \{v\}) - e(S)
= d_S(v) - d_S(u) - \mathbf{1}_{uv} \ge 0.
\end{equation}
More generally, let $T \subseteq D$ and $B \subseteq S$ with $|T| = |B|$. Comparing $S$ with $(S \setminus B) \cup T$ yields
\[
0 \le e((S \setminus B) \cup T) - e(S)
= -\sum_{u \in B} d_S(u) + e(B)
  + \sum_{v \in T} d_S(v) - e(T,B) + e(T),
\]
or equivalently,
\begin{equation}\label{eq:many-exchange-rearranged}
e(T,B) \le
\sum_{v \in T} d_S(v) - \sum_{u \in B} d_S(u) + e(T) + e(B).
\end{equation}
These identities follow by an exact count of edges lost and gained.

We use the following standard consequence of the finite hypergraph
Ramsey theorem; see~\cite{GrahamRothschildSpencer1990}.

\begin{lemma}[Graham, Rothschild and Spencer \cite{GrahamRothschildSpencer1990}]
\label{lem:simultaneous-ramsey}
Let integers $k\ge 2$ and $q\ge 1$ be given. There exists an integer $\rho_k(q)$ such that the following holds. Suppose on a common vertex set, we are given an $r$-uniform hypergraph for each uniformity $r=2,3,\dots,k$. Then every vertex subset of size at least $\rho_k(q)\ge q$ contains a $q$-element subset on which each of these hypergraphs is either complete or empty.
\end{lemma}

\begin{lemma}\label{lem:exact-independent}
Let \(G\) be a graph on \(n\) vertices and let \(t\) be an integer with
\(0<t<n\). Choose an \((n-t)\)-vertex subset \(S\subseteq V(G)\) that
minimizes \(e(G[S])\), and set \(D=V(G)\setminus S\). For every integer
\(k\ge2\) there exists a constant \(c_3=c_3(k)\) such that the following
holds: if for some integer \(x\) the set
\(\{v\in D : d_S(v)=x\}\)
contains an independent set $D'$ of order at least \(c_3\), then
\(s_k(G)\le |D|+2k(k-1)\).
\end{lemma}

\begin{proof}
Let $\rho_k$ be the constant supplied by Lemma~\ref{lem:simultaneous-ramsey}. 
Set $m_0=2k-2$ and, for $0\le i<2k-2$, set
\begin{equation}
    m_{i+1}=2\rho_k\!\left((k-1)2^{(k-1)(2k-3)}m_i\right).
    \label{eq:exact-thresholds}
\end{equation}
We prove the lemma with $c_3=m_{2k-2}$; clearly the sequence $m_0,m_1,\ldots,m_{2k-2}$ is nondecreasing.

Suppose, for a contradiction, that
\begin{equation}
    s_k(G)>|D|+2k(k-1).
    \label{eq:exact-contradiction}
\end{equation}
We now formalize the construction. A \emph{valid state} is a pair $(P,F)$ with
$\varnothing\ne P\subseteq D'$ and $F\subseteq S$ such that every vertex
of $F$ is either complete or anticomplete to $P$. For a valid state $(P,F)$, let
\[    
F^+=\{u\in F:P\subseteq N(u)\}
\]
and put $p=x-|F^+|$; the vertices of $F^+$ are called \emph{positive}. Note that the value of $p$ depends on the valid state $(P,F)$.

For every $u\in P$,
\begin{equation}
    d_{(S\setminus F)\cup P}(u)
    =d_{S\setminus F}(u)
    =x-|F^+|
    =p.
    \label{eq:candidate-degree}
\end{equation}
Indeed, $P$ is independent, every vertex of $F^+$ is adjacent to $u$, and
every vertex of $F\setminus F^+$ is nonadjacent to $u$. For
$v\in S\setminus F$, define its \emph{deficiency} by
\begin{equation}
    b_F(v)
    =p-d_{S\setminus F}(v)
    =x-|F^+|-d_S(v)+d_F(v).
    \label{eq:deficiency}
\end{equation}
A valid state $(P,F)$ is \emph{stable} if
$b_F(v)\ge0$ for every $v\in S\setminus F$, or equivalently, if every
vertex of $G[S\setminus F]$ has degree at most $p$.
The initial state $(D',\varnothing)$ is stable. Indeed, fix
$z\in D'$. Since $d_S(z)=x$, the one-vertex exchange
\eqref{eq:one-exchange} gives
$
    d_S(v)\le x
$
for every $v\in S$. Hence
$b_{\varnothing}(v)=x-d_S(v)\ge0$.

A finite sequence of valid states is called an \emph{admissible run} if it
starts at $(D',\varnothing)$ and every transition is obtained by one of
the following operations.

\begin{itemize}
    \item[(E)] If $(P,F)$ is stable and $|F^+|<2k-2$, choose
    $u\in S\setminus F$ and a nonempty set
    $P'\subseteq P\cap N(u)$ such that $b_F(u)\le k-2$, and replace
    $(P,F)$ by
    $
        (P',F\cup\{u\}).
    $
    We call $u$ an \emph{extension vertex}. Since $P'\subseteq P$,
    all vertices already in $F$ retain their complete or anticomplete
    status, while $u$ is complete to $P'$. Thus the new state is valid
    and has exactly one additional positive vertex.

    \item[(S)]  If $(P,F)$ is unstable and $|F^+|<2k-2$, let
$
    W=W(P,F)
    =
    \{w\in S\setminus F:
      d_{S\setminus F}(w)=\Delta(G[S\setminus F])\}.
$
For each subset $A\subseteq W$, define
$
    P_A=\{v\in P:N(v)\cap W=A\}.
$
The nonempty sets $P_A$, $A\subseteq W$, form a partition of $P$
into at most $2^{|W|}$ classes. Choose $A_0\subseteq W$ such that
$P_{A_0}$ has maximum size, and put
$
    P'=P_{A_0}.
$
Then
$
    |P'|\ge 2^{-|W|}|P|.
$
Moreover, for every $w\in W$, the vertex $w$ is complete to $P'$
if $w\in A_0$, and anticomplete to $P'$ if $w\notin A_0$.

First replace $(P,F)$ by $(P',F)$. Fix an arbitrary ordering
$w_1,\dots,w_{|W|}$ of $W$. Put $F_0=F$ and
\[
    F_j=F\cup\{w_1,\dots,w_j\},
    \qquad 1\le j\le |W|.
\]
Starting from $(P',F_0)$, pass successively through the states
$(P',F_1),\dots,(P',F_{|W|})$, stopping immediately if the number of
positive vertices reaches $2k-2$. Every state reached in this way is
included in the run.
\end{itemize}

A state is \emph{terminal} if $|F^+|=2k-2$, and no transition is made
from a terminal state. The sets $W$ arising in successive applications
of operation \textnormal{(S)} are called the \emph{stabilization sets}
and are denoted by $W_1,W_2,\dots$ in their order of selection. Every
transition only shrinks $P$ and enlarges $F$; consequently, a vertex
that is complete or anticomplete to the current set $P$ retains that
property in all later states.

\begin{claim}\label{cl:stabilization-bounds}
In every admissible run, at most \(2k-3\) stabilization sets are
selected before termination. Moreover, every stabilization set \(W\)
selected before termination satisfies $|W|\le k-1.$
\end{claim}
\begin{proof}
Let $W_1,W_2,\dots,W_s$ be the stabilization sets selected in an admissible run, in their order
of selection. For each \(i\in[s]\), let \((P_i,F_i)\) be the state immediately before \(W_i\) is selected, and put
$\mu_i=\Delta(G[S\setminus F_i]).$
Since for arbitrary \(i\in[s]\), \(W_i\) is selected at an unstable nonterminal state, there exists
\(w\in S\setminus F_i\) such that \(b_{F_i}(w)<0\). For this \(i\) and \(w\), we have
\(0 > b_{F_i}(w) = x-|F_i^+|-d_{S\setminus F_i}(w)\). By the definition of \(\mu_i\) we obtain
\[
\mu_i = \Delta(G[S\setminus F_i])
\ge d_{S\setminus F_i}(w)
> x-|F_i^+|
\ge x-(2k-3).
\]
On the other hand, \eqref{eq:one-exchange} gives \(d_S(v)\le x\) for
every \(v\in S\), and therefore \(x-2k+4\le\mu_i\le x\).

For \(i<s\), the set \(W_i\) is removed from the current induced subgraph on \(S\) before \(W_{i+1}\) is selected. Since \(W_i\) consists of all maximum-degree vertices of \(G[S\setminus F_i]\), deleting \(W_i\) strictly decreases the maximum degree, and subsequent operations only delete further vertices. Thus \(\mu_{i+1}<\mu_i\). Consequently \(\mu_1>\mu_2>\cdots>\mu_s\), where each \(\mu_i\) belongs to \(\{x-2k+4,x-2k+5,\dots,x\}\). This set has \(2k-3\) elements, so \(s\le2k-3\).

It remains to prove that every stabilization set has size at most \(k-1\). Suppose otherwise, and let \(W_t\) be the first stabilization set with \(|W_t|\ge k\). Then \(|W_j|\le k-1\) for every \(j<t\). Let \((P_t,F_t)\) be the state immediately before \(W_t\) is selected. All vertices of \(W_1,\dots,W_{t-1}\) belong to \(F_t\). Moreover, every extension vertex is positive when added and remains positive thereafter. Since \(W_t\) is selected before termination, at most \(2k-3\) extension vertices have been added. Therefore
\begin{align*}
    |F_t|&\le \sum_{j=1}^{t-1}|W_j|+(2k-3)
\le (t-1)(k-1)+(2k-3)\\
&\le (2k-4)(k-1)+(2k-3)
<2k(k-1).
\end{align*}
Thus \(|D\cup F_t| < |D|+2k(k-1) < s_k(G)\).
However, \(G-(D\cup F_t)=G[S\setminus F_t]\), and every vertex of
\(W_t\) attains the maximum degree in this graph. Since \(|W_t|\ge k\),
the set \(D\cup F_t\) is successful, a contradiction. Hence
\(|W|\le k-1\) for every stabilization set selected before termination.
\end{proof}

\begin{claim}
\label{cl:budget-classes}
For every state $(P,F)$ appearing in an admissible run before termination, $|F|\le k(2k-3)<2k(k-1).$
\end{claim}

\begin{proof}
Every vertex of \(F\) is either an extension vertex or belongs to a
stabilization set. Before termination, at most \(2k-3\) extension
vertices can be added. By Claim~\ref{cl:stabilization-bounds}, at most \(2k-3\)
stabilization sets are selected, each of size at most \(k-1\). Therefore
\(|F| \le (2k-3)(k-1)+(2k-3) = k(2k-3) < 2k(k-1)\).
\end{proof}

\begin{claim}\label{cl:stable-extension}
Let $(P,F)$ be a stable state appearing in an admissible run before
termination, and let $m\ge2k-2$. If $|P|\ge2\rho_k((k-1)m),$ then there exist $u\in S\setminus F$ and
$P_1\subseteq P\cap N(u)$ such that
$|P_1|\ge m$ and $b_F(u)\le k-2.$
Consequently, operation \textnormal{(E)} can be applied, and replacing
$(P,F)$ by $(P_1,F\cup\{u\})$ yields a valid state with exactly one
additional positive vertex. The new state need not be stable.
\end{claim}

\begin{proof}
For each $j\in\{0,1,\dots,k-1\}$, define a $(j+1)$-uniform hypergraph
$\mathcal H_j$ on $P$ by declaring a $(j+1)$-set to be an edge if it
is contained in $N(w)$ for some $w\in S\setminus F$ satisfying
$b_F(w)=j$.

We first show that every $k$-subset of $P$ contains an edge of at least
one of the hypergraphs $\mathcal H_0,\mathcal H_1,\dots,\mathcal H_{k-1}.$
Suppose that $Y\subseteq P$ is a $k$-set containing no such edge, and
put $Z_Y=(D\setminus Y)\cup F.$
Since $Y\subseteq D$ and $F\subseteq S$, the two parts of $Z_Y$ are
disjoint. By Claim~\ref{cl:budget-classes},
\begin{align*}
    |Z_Y|
    =|D|-|Y|+|F|=|D|-k+|F|<|D|+2k(k-1)<s_k(G),
\end{align*}
where the last inequality follows from
\eqref{eq:exact-contradiction}.
The remaining graph is
$
    G-Z_Y=G[(S\setminus F)\cup Y].
$
Since $D'$ is independent, every vertex of $Y$ has degree $p$ in this
graph by \eqref{eq:candidate-degree}. Let $w\in S\setminus F$. Stability
gives $b_F(w)\ge0$. If $b_F(w)\ge k$, then
$d_Y(w)\le|Y|= k\le b_F(w).$
If $0\le b_F(w)<k$ and $d_Y(w)>b_F(w)$, then
$Y\cap N(w)$ contains a $(b_F(w)+1)$-subset, which is an edge of
$\mathcal H_{b_F(w)}$, contrary to the choice of $Y$. Hence
$d_Y(w)\le b_F(w)$
in all cases. By \eqref{eq:deficiency},
\[
    d_{(S\setminus F)\cup Y}(w)
    =p-b_F(w)+d_Y(w)
    \le p.
\]
Thus all vertices of the remaining graph have degree at most $p$, while
the $k$ vertices of $Y$ have degree exactly $p$. Hence $Z_Y$ is
successful, contradicting $|Z_Y|<s_k(G)$.

We next extract a large homogeneous set. A $1$-uniform hypergraph on $P'$ is called complete if every singleton is an edge and empty if no singleton is an edge. Let $X_0$ be the set of vertices $v\in P$ such that $\{v\}$ is an edge of $\mathcal H_0$.
One of $X_0$ and $P\setminus X_0$ has size at least
$|P|/2\ge\rho_k((k-1)m)$; denote it by $P'$.
Restricting all hypergraphs to $P'$, the subhypergraph $\mathcal H_0[P']$ is either complete or empty.
Applying
Lemma~\ref{lem:simultaneous-ramsey} to
$\mathcal H_1,\dots,\mathcal H_{k-1}$, we obtain a set
$P_0\subseteq P'$ with
$
    |P_0|=(k-1)m
$
such that every $\mathcal H_j[P_0]$, $0\le j\le k-1$, is either complete
or empty. Since $|P_0|\ge k$ and every $k$-subset of $P_0$ contains an
edge of at least one $\mathcal H_j$, at least one of these hypergraphs
is complete on $P_0$. Let $\ell$ be the smallest index such that
$\mathcal H_\ell[P_0]$ is complete. Then
$\mathcal H_j[P_0]$ is empty for every $j<\ell$.

Fix an $\ell$-subset $A\subseteq P_0$, with $A=\varnothing$ when
$\ell=0$, and define
\[
    \mathcal W(A)
    =
    \{w\in S\setminus F:
      b_F(w)=\ell,\ A\subseteq N(w)\}.
\]
Suppose that
$
    |\mathcal W(A)|\ge k-\ell,
$
and put
$
    Z_A=(D\setminus A)\cup F.
$
Again, the two parts of $Z_A$ are disjoint, and
Claim~\ref{cl:budget-classes} gives
\begin{align*}
    |Z_A|
    =|D|-|A|+|F|=|D|-\ell+|F|<|D|+2k(k-1)<s_k(G).
\end{align*}

The remaining graph is
$
    G-Z_A=G[(S\setminus F)\cup A].
$
Since $D'$ is independent, every vertex of $A$ has degree $p$ in this
graph. Let $v\in S\setminus F$ and put $j=b_F(v)$. If $j<\ell$, then
$\mathcal H_j[P_0]$ is empty, so $v$ has at most $j$ neighbors in
$P_0$. Hence
$
    d_A(v)\le j=b_F(v).
$
If $j\ge\ell$, then
$
    d_A(v)\le|A|=\ell\le j=b_F(v).
$
Thus, in every case,
$
    d_A(v)\le b_F(v),
$
and consequently
\[
    d_{(S\setminus F)\cup A}(v)
    =p-b_F(v)+d_A(v)
    \le p.
\]

For every $w\in\mathcal W(A)$, we have
$b_F(w)=\ell$ and $A\subseteq N(w)$, so
$
    d_A(w)=\ell=b_F(w)
$
and therefore
$
    d_{(S\setminus F)\cup A}(w)=p.
$
It follows that the $\ell$ vertices of $A$, together with any
$k-\ell$ vertices of $\mathcal W(A)$, are $k$ vertices of maximum
degree in the remaining graph. Hence $Z_A$ is successful, contradicting
$|Z_A|<s_k(G)$. We conclude that
\begin{equation}
    |\mathcal W(A)|\le k-\ell-1.
    \label{eq:WA-bound}
\end{equation}

If $\ell=k-1$, then \eqref{eq:WA-bound} gives
$\mathcal W(A)=\varnothing$. Since $|P_0| = (k-1)m$ and $m \ge 2k-2$, it follows that $|P_0| > |A| = \ell = k-1$. Choose any $y\in P_0\setminus A$.
Since $\mathcal H_{k-1}[P_0]$ is complete, the set $A\cup\{y\}$ is
contained in the neighborhood of some $w\in S\setminus F$ with
$b_F(w)=k-1$. Hence $w\in\mathcal W(A)$, a contradiction. Therefore
$\ell\le k-2$.
Moreover, for every $y\in P_0\setminus A$, the completeness of
$\mathcal H_\ell[P_0]$ gives a vertex $w\in\mathcal W(A)$ adjacent to
$y$. Thus $\mathcal W(A)\ne\varnothing$ and
\[
    P_0\setminus A
    \subseteq
    \bigcup_{w\in\mathcal W(A)}
    \bigl(N(w)\cap(P_0\setminus A)\bigr).
\]
Consequently,
\[
    |P_0|-|A|
    \le
    \sum_{w\in\mathcal W(A)}
    |N(w)\cap(P_0\setminus A)|.
\]
Using \eqref{eq:WA-bound}, some $u\in\mathcal W(A)$ satisfies
\[
    |N(u)\cap(P_0\setminus A)|
    \ge
    \frac{(k-1)m-\ell}{k-\ell-1}.
\]
Since $A\subseteq N(u)$,
\begin{align*}
    |N(u)\cap P_0|\ge
    \frac{(k-1)m-\ell}{k-\ell-1}+\ell=m+\frac{\ell(m+k-\ell-2)}{k-\ell-1}\ge m.
\end{align*}
The last inequality holds since $0\le \ell\le k-2$ and $m\ge 2k-2$.
Set
$
    P_1=N(u)\cap P_0.
$
Then
\[
    P_1\subseteq P\cap N(u),
    \qquad
    |P_1|\ge m,
    \qquad
    b_F(u)=\ell\le k-2.
\]
Thus operation \textnormal{(E)} can be applied. Since $P_1$ is nonempty,
every old positive or anticomplete vertex retains its status, while
$u\notin F$ is complete to $P_1$. Hence
$(P_1,F\cup\{u\})$ is a valid state with exactly one additional
positive vertex.
\end{proof}

\begin{claim}\label{cl:stabilization}
Let $(P,F)$ be an unstable state appearing in an admissible run before
termination. The run can be extended, by repeatedly applying operation
\textnormal{(S)}, to a valid state $(P_1,F_1)$ with
$P_1\subseteq P$ and $F\subseteq F_1$ such that either $(P_1,F_1)$ is
terminal or it is stable. If $r$ stabilization sets are selected, then
$1\le r\le 2k-3$ and
\begin{equation}
    |P_1|\ge 2^{-(k-1)r}|P|
    \ge 2^{-(k-1)(2k-3)}|P|.
    \label{eq:stabilization-retention}
\end{equation}
Moreover, every vertex that becomes positive during these applications
had deficiency at most $k-1$ immediately before it was added to $F$.
\end{claim}

\begin{proof}
Starting from \((P,F)\), apply operation \textnormal{(S)} whenever the
current state is unstable and nonterminal. By
Claim~\ref{cl:stabilization-bounds}, at most \(2k-3\) stabilization sets
can be selected before termination. Hence this process stops after
finitely many applications, at a state \((P_1,F_1)\) that is either
terminal or stable.
Let \(W_1,\dots,W_r\) be the stabilization sets selected during the
process. Since the initial state is unstable, \(r\ge1\), while
Claim~\ref{cl:stabilization-bounds} gives \(r\le2k-3\). At each
application of \textnormal{(S)}, the current set \(P\) is partitioned
into at most \(2^{|W_i|}\) classes, and the largest class is retained.
As \(|W_i|\le k-1\), this retains at least a \(2^{-(k-1)}\) fraction of
the current set. Therefore
\begin{align*}
    |P_1| \ge (\prod_{i=1}^{r}2^{-|W_i|})|P|
\ge 2^{-(k-1)r}|P|
\ge 2^{-(k-1)(2k-3)}|P|.
\end{align*}

It remains to prove the deficiency bound. Suppose that a stabilization
set \(W\) is selected at the state \((Q,F_0)\), and write
$
W=\{w_1,\dots,w_s\}
$ in the order in which its vertices are added to \(F_0\). For
\(0\le i\le s\), let
$
F_i=F_0\cup\{w_1,\dots,w_i\}.
$
Let \(w_j\) be a vertex that becomes positive when it is added. Since
\(W\) is selected at an unstable state and every vertex of \(W\) has
degree \(\Delta(G[S\setminus F_0])\) in \(G[S\setminus F_0]\), we have
\[
\begin{aligned}
b_{F_0}(w_j)
=x-|F_0^+|-d_{S\setminus F_0}(w_j)
=x-|F_0^+|-\Delta(G[S\setminus F_0])
\le -1.
\end{aligned}
\]
For each \(1\le i<j\), the definition of deficiency gives
\[
b_{F_i}(w_j)-b_{F_{i-1}}(w_j)
=
-\bigl(|F_i^+|-|F_{i-1}^+|\bigr)
+\mathbf 1_{w_iw_j}.
\]
If \(w_i\) is positive, this difference is
$
-1+\mathbf 1_{w_iw_j}\le0;
$
if \(w_i\) is anticomplete to the retained class, it is
$
\mathbf 1_{w_iw_j}\le1.
$
Hence, in either case,
$
b_{F_i}(w_j)-b_{F_{i-1}}(w_j)\le1.
$
Therefore, immediately before \(w_j\) is added,
\[
\begin{aligned}
b_{F_{j-1}}(w_j)
\le b_{F_0}(w_j)+(j-1)
\le -1+(j-1)
\le -1+(|W|-1)
\le k-3
\le k-1,
\end{aligned}
\]
where the penultimate inequality follows from \(|W|\le k-1\).
Thus every vertex that becomes positive during the stabilization
process has deficiency at most (k-1) immediately before it is added.
\end{proof}

\begin{claim}\label{cl:terminal-run}
There exists an admissible run starting from $(D',\varnothing)$ and ending at a terminal state $(P^*,F^*)$ such that $|(F^*)^+|=2k-2$ and $|P^*|\ge m_0=2k-2.$ Moreover, every vertex of $(F^*)^+$ had deficiency at most $k-1$ immediately before it was added to $F$.
\end{claim}

\begin{proof}
We construct the run in rounds. Each round begins at a stable nonterminal state \((P,F)\), and we maintain the invariant
\begin{equation}    
|P|\ge m_{2k-2-|F^+|}.    
\label{eq:pool-invariant}
\end{equation}
The initial state \((D',\varnothing)\) is stable and satisfies \eqref{eq:pool-invariant}, since \(|D'|\ge m_{2k-2}\).
Suppose that a round begins at $(P,F)$, and set $\ell =2k-2-|F^+|.$ Since the state is nonterminal, $1\le \ell \le 2k-2$. By
\eqref{eq:pool-invariant} and \eqref{eq:exact-thresholds},
\[
    |P|\ge m_\ell
    =2\rho_k\!\left(
        (k-1)2^{(k-1)(2k-3)}m_{\ell-1}
      \right).
\]
Moreover, $2^{(k-1)(2k-3)}m_{\ell-1}\ge m_{\ell-1}\ge m_0=2k-2.$ Hence Claim~\ref{cl:stable-extension}, applied with $m=2^{(k-1)(2k-3)}m_{\ell-1},$ yields an application of operation \textnormal{(E)} and a valid state $(Q,F\cup\{u\})$ such that $|Q|\ge 2^{(k-1)(2k-3)}m_{\ell-1}.$ The vertex $u$ becomes positive and had deficiency at most $k-2$ immediately before it was added.

If $(Q,F\cup\{u\})$ is stable or terminal, set $(P',F')=(Q,F\cup\{u\}).$
Otherwise, apply Claim~\ref{cl:stabilization}, and let $(P',F')$ be
the resulting stable or terminal state. In either case,
\[
    |F'^+|\ge |F^+|+1
    \qquad\text{and}\qquad
    |P'|\ge m_{\ell-1}.
\]
Indeed, the second inequality is immediate if no stabilization is
needed. Otherwise, Claim~\ref{cl:stabilization} gives
\[
    |P'|
    \ge 2^{-(k-1)(2k-3)}|Q|
    \ge m_{\ell-1}.
\]
The same claim ensures that every vertex which becomes positive during
stabilization had deficiency at most $k-1$ immediately before it was
added.

If $(P',F')$ is terminal, the run stops. Otherwise, $(P',F')$ is
stable and nonterminal, and
\[
\begin{aligned}
    2k-2-|F'^+|
    &\le 2k-2-(|F^+|+1)=\ell-1.
\end{aligned}
\]
Since $(m_j)$ is nondecreasing,
\[
    |P'|
    \ge m_{\ell-1}
    \ge m_{2k-2-|F'^+|}.
\]
Thus \eqref{eq:pool-invariant} holds at the beginning of the next
round.

Every nonterminal round increases the number of positive vertices by
at least one. Hence, after at most $2k-2$ rounds, the run reaches a
terminal state $(P^*,F^*)$. By definition, $|(F^*)^+|=2k-2.$
Applying the estimate above to the final round, and noting that its
corresponding index satisfies $i\ge1$, we obtain
\[
    |P^*|\ge m_{i-1}\ge m_0=2k-2.
\]
Finally, every vertex of $(F^*)^+$ was added either by operation
\textnormal{(E)}, with deficiency at most $k-2$, or during an
application of operation \textnormal{(S)}, with deficiency at most
$k-1$. This proves the claim.
\end{proof}

Let $(P^*,F^*)$ be the terminal state given by
Claim~\ref{cl:terminal-run}. List the vertices of $(F^*)^+$ in the
order in which they were added as $u_1,\dots,u_{2k-2}$. For each
$i\in[2k-2]$, let $(P_i,F_i)$ be the state immediately before $u_i$
was added. Then $F_i^+=\{u_1,\dots,u_{i-1}\}.$

By \eqref{eq:deficiency} and the preceding deficiency bound,
\[
\begin{aligned}
    b_{F_i}(u_i)
    &=x-d_S(u_i)+d_{F_i\setminus F_i^+}(u_i)
      +d_{F_i^+}(u_i)-(i-1)\le k-1.
\end{aligned}
\]
Since $d_{F_i\setminus F_i^+}(u_i)\ge0$, we obtain
\[
    x-d_S(u_i)
    \le k-1+(i-1)-d_{F_i^+}(u_i).
\]
Moreover,
$
    \sum_{i=1}^{2k-2}d_{F_i^+}(u_i)=e((F^*)^+),
$
because every edge of $G[(F^*)^+]$ is counted exactly once, namely
when its later-added endpoint is considered. Summing the preceding
inequality over $i$ gives
\begin{equation}
\sum_{i=1}^{2k-2}\bigl(x-d_S(u_i)\bigr)
\le (2k-2)(k-1)+\binom{2k-2}{2}-e((F^*)^+).
\label{eq:positive-upper}
\end{equation}

Since $|P^*|\ge2k-2$, choose a set $T\subseteq P^*$ with
$|T|=2k-2$. As $P^*\subseteq D'$, the set $T$ is independent and
$d_S(v)=x$ for every $v\in T$. Moreover, every vertex of $(F^*)^+$
is complete to $P^*$, and hence
\[
    e(T,(F^*)^+)=(2k-2)^2.
\]
Applying \eqref{eq:many-exchange-rearranged} with
$B=(F^*)^+$, and using $e(T)=0$, yields
\begin{align*}
(2k-2)^2
&=e(T,(F^*)^+)\le \sum_{v\in T}d_S(v)
   -\sum_{i=1}^{2k-2}d_S(u_i)
   +e(T)+e((F^*)^+)\\
&=\sum_{i=1}^{2k-2}\bigl(x-d_S(u_i)\bigr)
  +e((F^*)^+).
\end{align*}
Together with \eqref{eq:positive-upper}, this gives
\[
    (2k-2)^2
    \le (2k-2)(k-1)+\binom{2k-2}{2}
    =(2k-2)^2-(k-1),
\]
which is impossible as $k\ge2$. Hence
\eqref{eq:exact-contradiction} is false, and therefore
$s_k(G)\le |D|+2k(k-1).$
\end{proof}

\begin{corollary}
\label{cor:bounded-multiplicity}
Let \(k\ge 2\) be fixed, and let \(c_3 = c_3(k)\) be the constant from Lemma~\ref{lem:exact-independent}. 
Suppose \(G\) is a graph with \(s_k(G) > |D| + 2k(k-1)\), where \(S\) and \(D\) arise from the minimum-edge construction for some choice of \(t\) with \(0<t<|V(G)|\).
Then for every integer \(x\), 
\(\bigl|\{v\in D : d_S(v) = x\}\bigr| < R(k, c_3)\).
\end{corollary}

\begin{proof}
Assume, for contradiction, that for some $x$ the set $X = \{v\in D : d_S(v) = x\}$ satisfies $|X| \ge R(k, c_3)$.
By definition of the Ramsey number, there exists $X_1 \subseteq X$ such that $G[X_1]$ is either an independent set of order $c_3$ or a clique of order $k$.

If $G[X_1]$ is independent, Lemma~\ref{lem:exact-independent} gives $s_k(G)\le |D|+2k(k-1)$, contradicting $s_k(G) > |D|+2k(k-1)$.
If $G[X_1]$ is a clique of order $k$, consider the graph $G[S\cup X_1]$.
Every $u\in X_1$ satisfies $d_{S\cup X_1}(u) = d_S(u) + d_{X_1}(u) = x + k-1$.
For any $v\in S$, we have two possibilities:
\begin{itemize}
    \item If $v$ has a neighbour in $X_1$, then by \eqref{eq:one-exchange}, $d_S(v) \le x-1$. Since $|X_1|=k$, we get $d_{X_1}(v)\le k$, and hence
    $d_{S\cup X_1}(v) \le (x-1) + k = x+k-1$.
    \item If $v$ is adjacent to no vertex of $X_1$, then \eqref{eq:one-exchange} yields $d_S(v) \le x$, so $d_{S\cup X_1}(v) = d_S(v) \le x \le x+k-1$.
\end{itemize}
Thus every vertex of $S\cup X_1$ has degree at most $x+k-1$, while all vertices of $X_1$ attain this value.
Consequently, the maximum degree of $G[S\cup X_1]$ is $x+k-1$, achieved by the $k$ vertices of $X_1$.
Hence, $D\setminus X_1$ is a successful deletion set, contradicting $s_k(G) > |D| + 2k(k-1)$.
\end{proof}

\begin{lemma}
\label{lem:height-width}
Fix an integer \(k\ge2\). Let \(G\) be a graph on \(n\) vertices and let \(t\) be an integer with \(0<t<n\). 
Choose an \((n-t)\)-vertex subset \(S\subseteq V(G)\) that minimizes \(e(G[S])\), and set \(D=V(G)\setminus S\).
Let \(c_2 = c_2(k)\) be the constant from Lemma~\ref{lem:degree-balancing}, 
let \(a,b\) be integers with \(b\ge 0\), and suppose \(W\subseteq D\) satisfies 
\(|W|\ge R(k,k)\) and \(a\le d_S(v)\le a+b\) for all \(v\in W\). 
If \(a-\min_{v\in D}d_S(v)\ge c_2 b+k\), then \(s_k(G)\le |D|-k+c_2 b\).
\end{lemma}

\begin{proof}
Since $|W|\ge R(k,k)$, Ramsey's theorem yields a \(k\)-set \(X\subseteq W\) such that \(G[X]\) is either a clique or an independent set. In either case, \(G[X]\) is regular.

For any \(u,v\in X\), we have \(|d_S(u)-d_S(v)|\le b\) because \(X\subseteq W\), and \(d_X(u)=d_X(v)\) by the regularity of \(G[X]\). Thus \(|d_{S\cup X}(u)-d_{S\cup X}(v)|\le b\). By Lemma~\ref{lem:degree-balancing}, there exists \(U\subseteq S\) with \(|U|\le c_2b\) such that all vertices of \(X\) have the same degree, say \(p\), in \(G[(S\setminus U)\cup X]\). Since \(d_S(u)\ge a\) for every \(u\in X\), we have \(p\ge a-|U|\ge a-c_2b\).

Choose \(v_0\in D\) with \(d_S(v_0)=\min_{v\in D}d_S(v)\). By~\eqref{eq:one-exchange}, \(d_S(u)\le d_S(v_0)\) for every \(u\in S\). Hence, for each \(u\in S\setminus U\),
\[
d_{(S\setminus U)\cup X}(u)
=d_{S\setminus U}(u)+d_X(u)
\le d_S(u)+k
\le d_S(v_0)+k.
\]
The assumption \(a-d_S(v_0)\ge c_2b+k\) gives \(p\ge a-c_2b\ge d_S(v_0)+k\). Therefore, every vertex of \(X\) has degree \(p\), while every vertex of \(S\setminus U\) has degree at most \(p\), in \(G[(S\setminus U)\cup X]\). Thus its maximum degree is attained by all \(k\) vertices of \(X\).

Finally, \(G[(S\setminus U)\cup X]\) is obtained by deleting \((D\setminus X)\cup U\), and
\[
\bigl|(D\setminus X)\cup U\bigr|
=|D|-k+|U|
\le |D|-k+c_2b.
\]

\end{proof}

\begin{lemma}\label{lem:quadratic-spread}
For every integer $k\ge2$, there exists an integer
$r_0=r_0(k)$ such that, for every integer $r\ge r_0$
and every graph $G$ satisfying $s_k(G)>r$, we have
\[
\Delta(G)\ge
\frac{r^2}{64c_2(k)R(k,k)}.
\]
\end{lemma}

\begin{proof}
Let $c_2=c_2(k)$ and let $c_3=c_3(k)$ be the constants from Corollary~\ref{cor:bounded-multiplicity} and Lemma~\ref{lem:height-width}.
Define
\[
i_0 = 1+\bigl\lceil 2k(R(k,c_3)-1)/R(k,k)\bigr\rceil .
\]
Choose \(r_0=r_0(k)\) large enough so that for every integer \(r\ge r_0\) the following hold simultaneously:
\begin{itemize}
    \item[(i)] \(r \ge \max\{4k,\, 2R(k,k)\}\);
    \item[(ii)] \(\lfloor r/2\rfloor + 2k(k-1) \le r\);
    \item[(iii)] \(i_0 + \Bigl\lceil \log_{1+\frac{1}{2c_2}}\bigl(\frac{r}{4k}\bigr)\Bigr\rceil \;\le\; \frac12\Bigl\lfloor \frac{\lfloor r/2\rfloor}{R(k,k)}\Bigr\rfloor\).
\end{itemize}
For fixed $k$, the left-hand side of (iii) is $O_k(\log r)$, whereas the right-hand side is $\Theta_k(r)$. Hence an integer $r_0=r_0(k)$ satisfying (i)--(iii) exists.

Now fix any \(r\ge r_0\) and any graph \(G\) with \(s_k(G)>r\).
Since no deletion set of size at most \(r\) can be successful, we must have \(|V(G)|\ge r+k\); otherwise removing all but \(k-1\) vertices gives a successful set of size at most \(r\).  Hence \(|V(G)| > \lfloor r/2\rfloor\).

Set \(t = \lfloor r/2\rfloor\). Choose an \((|V(G)|-t)\)-vertex subset \(S\subseteq V(G)\) that minimizes \(e(G[S])\), and let \(D = V(G)\setminus S\); thus \(|D| = t = \lfloor r/2\rfloor\).
Label the vertices of \(D\) as \(v_1,\dots,v_{|D|}\) with \(d_S(v_1)\le\cdots\le d_S(v_{|D|})\).
Let \(m = \lfloor |D|/R(k,k)\rfloor\).  For \(i=1,\dots,m\), define $$
D_i = \{v_{(i-1)R(k,k)+1},\dots,v_{iR(k,k)}\},
h_i = d_S(v_{(i-1)R(k,k)+1})-d_S(v_1),$$
$$w_i = d_S(v_{iR(k,k)})-d_S(v_{(i-1)R(k,k)+1}).$$
Then for every \(v\in D_i\) we have \(d_S(v_1)+h_i \le d_S(v) \le d_S(v_1)+h_i+w_i\).

Condition (ii) gives \(|D|+2k(k-1)\le r\).  Together with Corollary~\ref{cor:bounded-multiplicity} and \(s_k(G)>r\), this implies that each integer occurs fewer than \(R(k,c_3)\) times among \(d_S(v_1),\dots,d_S(v_{|D|})\).

We claim that for every \(i\in [m]\),
\begin{equation}
w_i>\frac{r}{4c_2}\quad\text{or}\quad h_i < c_2 w_i + k .
\label{eq:block-alternative}
\end{equation}
Indeed, if some \(j\) had \(w_j\le r/(4c_2)\) and \(h_j\ge c_2 w_j+k\), then Lemma~\ref{lem:height-width} with \(a=d_S(v_1)+h_j\) and \(b=w_j\) would give a successful deletion set of size at most \(|D|-k+c_2 w_j \le |D|-k+r/4 < r\), contradicting \(s_k(G)>r\).

Because each degree value repeats at most \(R(k,c_3)-1\) times, at most \(2k(R(k,c_3)-1)\) vertices of \(D\) satisfy \(d_S(v)-d_S(v_1) < 2k\).  By condition (iii) we have \(m \ge 2i_0\), and for all \(i\ge i_0\)
\begin{equation}
h_i \ge 2k . \label{eq:h2k}
\end{equation}

The sequence \(d_S(v_j)\) is nondecreasing, so \(h_{i+1}\ge h_i+w_i\) for \(1\le i<m\).
Take any \(\ell\) with \(i_0\le\ell< m\) and suppose \(h_\ell < r/2\).
If the first alternative of \eqref{eq:block-alternative} holds, then \(w_\ell > r/(4c_2) \ge h_\ell/(2c_2)\).
If the second holds, then \(w_\ell > (h_\ell-k)/c_2 \ge h_\ell/(2c_2)\) by \eqref{eq:h2k}.
Thus in either case
\begin{equation}
h_{\ell+1} \ge h_\ell + w_\ell > \Bigl(1+\frac{1}{2c_2}\Bigr)h_\ell . \label{eq:geometric}
\end{equation}

Let
\[
L_r=
\left\lceil
\log_{1+\frac1{2c_2}}
\left(\frac r{4k}\right)
\right\rceil.
\]
If $h_{i_0}\ge r/2$, take $j=i_0$.
Otherwise, as long as $h_\ell<r/2$, inequality
\eqref{eq:geometric} applies. Hence
\[
h_{i_0+s}\ge
\left(1+\frac1{2c_2}\right)^s h_{i_0}
\]
for all relevant $s$. Since $h_{i_0}\ge2k$, after at most
$L_r$ steps we obtain an index
$j\le i_0+L_r$ with $h_j\ge r/2$.
Condition~(iii) ensures that all these indices are at most
$\lfloor m/2\rfloor$.

Now for every \(m\ge q\ge j\), monotonicity yields \(h_q\ge r/2\).  For such \(q\), \eqref{eq:block-alternative} together with \(r\ge 4k\) gives \(w_q > r/(4c_2)\): either directly, or via \(w_q > (h_q-k)/c_2 \ge (r/2-k)/c_2 \ge r/(4c_2)\).  Because \(j\le \lfloor m/2\rfloor\), at least \(m/2\) of the blocks \(D_i\) satisfy \(w_i > r/(4c_2)\).

All \(w_i\) are non‑negative by construction, hence
\begin{align*}
d_S(v_{mR(k,k)})-d_S(v_1)\ge\sum_{i=1}^m(d_S(v_{iR(k,k)})-d_S(v_{(i-1)R(k,k)+1}))
= \sum_{i=1}^{m} w_i \ge \frac{m}{2}\cdot\frac{r}{4c_2}.
\end{align*}
Since \(|D|\ge mR(k,k)\) and the degrees are sorted,
\[
d_S(v_{|D|})-d_S(v_1) \ge d_S(v_{mR(k,k)})-d_S(v_1) \ge \frac{m}{2}\cdot\frac{r}{4c_2}.
\]
Since $|D|\ge R(k,k)$, we have
$
m=\left\lfloor\frac{|D|}{R(k,k)}\right\rfloor
\ge\frac{|D|}{2R(k,k)}.
$
Therefore,
\[
d_S(v_{|D|})-d_S(v_1) \ge \frac{1}{2}\cdot\frac{|D|}{2R(k,k)}\cdot\frac{r}{4c_2}
\ge \frac{r/4}{4R(k,k)}\cdot\frac{r}{4c_2} = \frac{r^2}{64\,c_2\,R(k,k)},
\]
where we used \(|D| = \lfloor r/2\rfloor \ge r/4\).
Finally, \(d_S(v_{|D|})\le\Delta(G)\) and \(d_S(v_1)\ge0\) imply \(\Delta(G)\ge \frac{r^2}{64\,c_2\,R(k,k)}\), as required.
\end{proof}
\begin{proof}[Proof of Theorem~\ref{thm:main}]
Fix an integer $k\ge2$, and let $G$ be a graph on at least $R(k,k)$ vertices.
Let \(r_0=r_0(k)\) be the constant from Lemma~\ref{lem:quadratic-spread}.
If \(\Delta(G)=0\), then the empty set is successful, so \(s_k(G)=0\) and the bound holds trivially.
Hence we may assume \(\Delta(G)\ge 1\).
If \(s_k(G)\le r_0\), then \(s_k(G)\le r_0\le r_0\sqrt{\Delta(G)}\) because \(\sqrt{\Delta(G)}\ge 1\).
Otherwise \(s_k(G)\ge r_0+1\).  Set \(r=s_k(G)-1\); then \(r\ge r_0\) and \(s_k(G)>r\).
Lemma~\ref{lem:quadratic-spread} yields
\(\Delta(G)\ge \frac{r^2}{64c_2(k)R(k,k)} = \frac{(s_k(G)-1)^2}{64c_2(k)R(k,k)}\).
Solving for \(s_k(G)\) gives
\[
s_k(G)\le 1+\sqrt{64c_2(k)R(k,k)}\sqrt{\Delta(G)}\le\bigl(1+\sqrt{64c_2(k)R(k,k)}\bigr)\sqrt{\Delta(G)}.
\]
Therefore in all cases,
\(s_k(G)\le \max\{r_0,\;1+\sqrt{64c_2(k)R(k,k)}\}\sqrt{\Delta(G)}\).
When \(|V(G)|\ge R(k,k)\) we have \(s_k(G)=f_k(G)\), completing the proof.
\end{proof}

\section*{Declaration on the Use of Generative AI}
The authors used ChatGPT 5.6 Pro to assist in discussing proof strategies, checking proofs, and
improving exposition.

\end{document}